\documentclass[oneside, english]{amsart}
\usepackage{color}
\usepackage{amsthm}
\usepackage{graphicx}
\usepackage{amssymb, eucal}
\usepackage{tikz}

\makeatletter

\numberwithin{equation}{section} 
\numberwithin{figure}{section} 

\theoremstyle{plain}
\theoremstyle{definition}
\newtheorem{defn}{Definition}
\newtheorem*{defn*}{Definition}
\theoremstyle{plain}
 \newtheorem*{thm*}{Theorem}
\theoremstyle{plain}
\newtheorem{thm}{Theorem}[section]
\theoremstyle{definition}
 \newtheorem{example}[thm]{Example}
 \theoremstyle{plain}
 \newtheorem{lem}[thm]{Lemma}
 \theoremstyle{plain}
 \newtheorem{cor}[thm]{Corollary}
 \theoremstyle{plain}
 \newtheorem{prop}[thm]{Proposition}
 \theoremstyle{definition}
 \newtheorem{xca}[thm]{Exercise}

\@ifundefined{definecolor}
{\usepackage{color}}{}
\@ifundefined{definecolor}{\@ifundefined{definecolor}
{\usepackage{color}}{}
}{}\makeatother

\newcommand{\bA}{\ensuremath{\mathbf{A}}}
\newcommand{\bB}{\ensuremath{\mathbf{B}}}
\newcommand{\bC}{\ensuremath{\mathbf{C}}}
\newcommand{\bS}{\ensuremath{\mathbf{S}}}

\newcommand{\bT}{\ensuremath{\mathbf{T}}}
\newcommand{\cN}{\ensuremath{\mathcal{N}}}
\newcommand{\cT}{\ensuremath{\mathcal{T}}}
\newcommand{\cI}{\ensuremath{\mathcal{I}}}
\newcommand{\cP}{\ensuremath{\mathcal{P}}}

\makeatletter
\def\Ddots{\mathinner{\mkern1mu\raise\p@
\vbox{\kern7\p@\hbox{.}}\mkern2mu
\raise4\p@\hbox{.}\mkern2mu\raise7\p@\hbox{.}\mkern1mu}}
\makeatother

\begin{document}
\title{From indexed grammars to generating functions}
\author{Jared Adams}
\address{Department of Mathematics, Southern Utah University, Cedar City, UT, USA 84720}
\author{Eric Freden}
\address{Department of Mathematics, Southern Utah University, Cedar City, UT, USA 84720}
\author{Marni Mishna}
\address{Department of Mathematics, Simon Fraser University, Burnaby
 BC,  Canada V5A 1S6 }
\thanks{The third author gratefully acknowledges the support of NSERC Discovery
 Grant funding (Canada), and LaBRI (Bordeaux) for hosting during the
 completion of the work.}
\date{8 August 2013}
\begin{abstract}
We extend the DSV method of computing the growth
series of an unambiguous context-free language to the larger class
of indexed languages. We illustrate the technique with numerous examples. 
\end{abstract}
\maketitle

\section{Introduction}
\subsection{Indexed grammars}
Indexed grammars were introduced in the thesis of Aho in the late
1960s to model a natural subclass of context-sensitive languages, more
expressive than context-free grammars with interesting closure
properties~\cite{Aho1968,HoUl1979}. The original reference for basic
results on indexed grammars is~\cite{Aho1968}.  The complete
definition of these grammars is equivalent to the following reduced
form.

\begin{defn}\label{def:IG}
A \emph{reduced indexed grammar\/} is a 5-tuple $(\cN,\cT,\cI,\cP,\bS)$,
such that
\begin{enumerate}
\item $\cN$, $\cT$ and $\cI$ are three mutually disjoint finite sets
 of symbols: the set~$\cN$ of \emph{non-terminals} (also called
\emph{variables}), $\cT$ is the set of \emph{terminals} and~$\cI$ is the
 set of \emph{indices} (also called \emph{flags});
\item $\bS\in \cN$ is the start symbol;
\item $\cP$ is a finite set of productions, each having the form of
 one of the following: 
\begin{enumerate}
\item $\bA\rightarrow\alpha$ \label{plain}
\item $\bA\rightarrow\bB_{f}$\qquad\emph{(push)}\label{push}
\item $\bA_{f}\rightarrow\beta$\,\,\qquad\emph{(pop)} \label{pop}
\end{enumerate}
where $\bA,\bB\in \cN$, $f\in \cI$ and
$\alpha,\beta\in(\cN\cup\cT)^{*}$.
\end{enumerate}
Observe the similarity to context-free grammars which are only defined
by production rules of type~\eqref{plain}. The language defined by an
indexed grammar is the set of all strings of terminals that can be
obtained by successively applying production rules begining with the
rule that involves the start symbol $\bS$. A key distinction from
context free grammars is that rather than expand non-terminals, we
expand non-terminal/stack pairs:~$(\bA, \iota)$, $\iota\in \cI^*,
\bA\in\cN$.  Here, the start symbol $\bS$ is shorthand for the pair
$(\bS, \epsilon)$, where $\epsilon$ denotes the empty stack.

Production rules in~$\cP$ are interpreted as follows.  The stack is
implicit, and is copied when the production is applied. For example,
the type~\eqref{plain} production rule $\bA\rightarrow a\bB\bC$ is 
shorthand for $ (\bA,\iota)\rightarrow a(\bB, \iota)(\bC, \iota)$, 
for~$\bA, \bB,\bC\in\cN$, $a\in\cT$ and $\iota\in \cI^*$.

A production rule of form~\eqref{push} encodes a push onto the stack,
and a production rule of the form~\eqref{pop} encodes a pop off of the
stack.  For example, the production rule $\bA\rightarrow\bB_f$ applied
to $(\bA, \iota)$ expands to $(\bB, \iota')$ where $\iota'$ is the
stack $\iota$ with the character $f$ pushed on. Likewise,
$\bA_{f}\rightarrow \beta$ can only be applied to $(\bA, \iota)$ if
the top of the stack string~$\iota$ is $f$. The result is $\beta$ such
that any nonterminal $\bB \in \beta$ is of the form $(\bB, \iota'')$,
where $\iota''$ is the stack $\iota$ with the top character popped
off.
\end{defn}

To lighten the notation the stack is traditionally written as a
subscript. Note the difference: the presence of a subscript in a
production rule is shorthand for an infinite collection of production
rules, whereas in a derivation the stack is viewed as part of the symbol.
Furthermore, it is also useful to introduce an end of stack symbol,
which we write $\$$. This symbol is reserved strictly for the last
position in the stack. This permits us to expand a non-terminal into a
terminal only when the stack is empty.  These subtleties are best made
clear through an example.

\begin{example}
The class of indexed languages is strictly larger than the class of context-free
languages since it contains the language
$\mathfrak{L}=\{a^{n}b^{n}c^{n}:n>0\}$. This language is generated by the indexed
grammar $(\{\bS, \bT, \bA,\bB, \bC\}, \{a,b,c\}, \{f\}, \cP,\bS)$ with 
\begin{multline*}
\cP=\left\{ 
\bS\rightarrow \bT_\$,\quad 
\bT\to \bT_f,\quad 
\bT\to \bA\bB\bC,\right.\\\left.
\bA_f\to a\bA,\quad \bA_\$\to a,\quad
\bB_f\to b\bB,\quad \bB_\$\to b,\quad
\bC_f\to c\bC,\quad \bC_\$\to c \right\}.
\end{multline*}
A typical derivation is as follows. We begin with $\bS$ and derive $aaabbbccc$: 
\begin{multline*}
\bS\to \bT_\$  \to \bT_{f\$} 
\to \bT_{ff\$} \to
\bA_{ff\$}\bB_{ff\$}\bC_{ff\$}\to a  \bA_{f\$}\bB_{ff\$}\bC_{ff\$}\to \\[2mm]
aa\bA_{\$}\bB_{ff\$}\bC_{ff\$} \to aaa\bB_{ff\$}\bC_{ff\$} 
\to aaab\bB_{f\$}\bC_{ff\$}\to\dots\to\,  aaabbbccc.
\end{multline*}
The generic structure of a derivation has the following three phases: First,
there is an initial push phase to build up the index string; This is followed
by a transfer stage where the stack is copied;  Finally, there is a
a pop stage which converts indices into terminal symbols. Most of
our examples have this same structure.

This particular grammar is easy to generalize to a language with some
fixed number~$k$ of symbols repeated~$n$ times:~$\{a_1^na_2^n\dots a_k^n:
n>0\}$.
\end{example}
In the case of reduced grammars, at most one index symbol is loaded or
unloaded in any production. We use two additional properties which do
not restrict the expressive power. A grammar is in \emph{strongly
 reduced form} as per our definition and if there are no useless
non-terminals. That is, every non-terminal $\mathbf{V}$ satisfies both
$S\overset{*}{\to}\alpha\mathbf{V}_{\sigma}\alpha'$ and
$\mathbf{V}_{\sigma}\overset{*}{\to}w$ for $\sigma\in\cI^*$,
$w\in\cT^*$ and $\alpha, \alpha'\in(\cT+\cN)^*$. A grammar is
$\varepsilon$-\textit{free} if the only production involving the empty
string $\varepsilon$ is $\mathbf{S}\rightarrow\varepsilon$.  Indexed
grammars $\mathcal{\mathfrak{G}}_{1}$ and $\mathcal{\mathfrak{G}}_{2}$
are \textit{equivalent} if they produce the same language
$\mathcal{\mathfrak{L}}$.

\begin{thm*}[\cite{PaDu1990}]
Every indexed grammar $\mathfrak{G}$ is equivalent to some strongly reduced,
$\varepsilon$-free grammar $\mathfrak{G}'$. Furthermore, there is an effective 
algorithm to convert $\mathfrak{G}$ to $\mathfrak{G}'$.
\end{thm*}
Consequently we can assume all grammars are already strongly reduced
(most of our examples are). On the other hand, we have found that
$\varepsilon$-productions are a useful crutch in designing grammars
(several of our examples are not $\varepsilon$-free).

\subsection{The set of indexed languages}
The set of all languages generated by indexed grammars forms the set
of indexed languages.  As alluded to above, this is a full abstract
family of languages which is closed under union, concatenation, Kleene
closure, homomorphism, inverse homomorphism and intersection with
regular sets. The set of indexed languages, however is not closed
under intersection or complement. The standard machine type that
accepts the class of indexed languages is the nested stack automaton.

This class of languages properly includes all context-free
languages. These are generated by grammars such that $\cI$ is empty.
One way to view indexed grammars is as an extension of context-free
grammars with an infinite number of non-terminals, however the
allowable productions are quite structured.  Furthermore, it is a
proper subset of the class of context-sensitive languages. For
instance $\left\{ \left(ab^{n}\right)^{n}:\ n\geq0\right\} $ is
context-sensitive but not indexed~\cite{Gilm1996}.

Formal language theory in general and indexed languages in particular
have applications to group theory. Two good survey articles are
~\cite{Rees1999} and~\cite{Gilm2005}. Bridson and
Gilman~\cite{BrGi1996} have exhibited indexed grammar combings for
fundamental $3$-manifold groups based on Nil and Sol geometries (see
Example~\ref{ex:BG} below).  More recently~\cite{HoRo2006} showed that
the language of words in the standard generating set of the Grigorchuk
group that do not represent the identity (the so-called \emph{co-word}
problem) forms an indexed language. The original DSV method
(attributed to Delest, Sch\"utzenberger, and Viennot~\cite{Del1996})
of computing the growth of a context-free language was successfully
exploited~\cite{FrSc2008} to compute the algebraic but non-rational
growth series of a family of groups attributed to Higman. One of our
goals is to extend this method to indexed grammars to deduce results
on growth series.

\subsection{Ordinary generating functions}
\label{sec:ogf}
Generating functions are well-suited to answer  enumerative questions
about languages over finite alphabets, in particular, the number of
words of a given length. For any language $\mathfrak{L}$ with finite alphabet
$\cT$, let $L_{n}$ be the number of words of length $n$. The
\emph{ordinary generating function} of the language is the formal power
series $L(z)=\sum_{n\geq0}L_{n}z^{n}$.  We use this terminology
interchangeably with \emph{growth series}. Note that each $L_n$ is bounded
by the number of words of
length $n$ in the free monoid $\cT^*$. Consequently, the
generating function $L(z)$ has a positive radius of convergence.

One motivation for our study is to understand the enumerative nature
of classes of languages beyond context-free, and simultaneously to
understand the combinatorial nature of function classes beyond
algebraic. To be more precise, it is already known that regular
languages have generating functions that can be expressed as rational
functions, i.e. the Taylor series of a function $P(z)/Q(z)$ where $P$
and $Q$ are polynomials. Furthermore, unambiguous context free
languages have algebraic generating functions, that is, they satisfy
$P(L(z),z)$ for some bivariate polynomial $P(x,y)$ with integer
coefficients. This kind of generating function property has
consequences on the asymptotic form, and can be used to exclude
languages from the companion classes, by analytic arguments on the
generating function. For example, Flajolet~\cite{Flaj1987} proves the
inherent ambiguity of several context-free languages by demonstrating
the transcendence of their generating functions.

There are two natural contenders for function classes that may capture
indexed grammars: \emph{D-finite} and \emph{differentiably 
 algebraic}. A series is said to be \emph{D-finite} if it satisfies a
homogeneous linear differential equation with polynomial
coefficients. A series $L(z)$ is said to be \emph{differentiably
 algebraic} if there is a non-trivial $k+1$-variate polynomial
$P(x_{0},x_{1},\dots,x_{k})$ with integer coefficients such that
\[P\left(
 L(z),\frac{d}{dz}L(z),\frac{d^{2}}{dz^{2}}L(z),\dots,\frac{d^{k}}{dz^{k}}L(z)\right)
\equiv0.\] 
We prove that neither of these classes capture indexed 
grammars. In fact, many of our examples of indexed grammars have
lacunary generating functions, with a natural boundary at the unit
circle, because they are so sparse. This is perhaps unsatisfying, but
it also illustrates a key difference between computational complexity
and analytic complexity; a distinction which is not evident after
studying only context-free and and regular languages. 

That said, the expressive power of growth series derived from indexed
languages has been broached by previous authors. In~\cite{LK2003}, the
authors consider a limitation on possible productions, which is close
to one of the restrictions we consider below. They are able to
describe the recurrence type satisfied by all sequences $u(n)$ for
which the corresponding language~$\{a^{u(n)}\}$ is generated by this
restricted set of indexed grammars.

Furthermore, we mention that other characterizations of indexed
languages are equally amenable to analysis. In particular, indexed
languages are equivalent to sequences of level 2 in the sense
of~\cite{Seni2007}, and have several different descriptions. Notably,
they satisfy particular systems of catenative recurrent relations, to
which methods comparable to what we present here may apply. By taking
combinations of such sequences, Fratani and
Senizergues~\cite{FrSe2005} can give a characterization of D-finite
functions with rational coefficients. 

A second motivation for the present work is to verify that various growth
rates are achievable with an indexed grammar. To that end, we often start
with a language where the enumeration is trivial, but the challenge is
to provide an indexed grammar that generates it. Our
techniques verify that the desired growth rate has been achieved.

\subsection{Summary}
Ideally, we would like to describe an efficient algorithm to determine
the generating function of an indexed language given only a
specification of its indexed grammar.  Towards this goal we first
describe a process in the next section that works under some
conditions including only one stack symbol (excluding the end of stack
symbol).
Proposition~\ref{prop:push-to-inf II} summarizes the conditions, and the
results. We have several examples to illustrate the procedure. In
Section~\ref{sec:multiple} this is generalized to multiple stack
symbols that are pushed in order. In this section we also illustrate
the inherent obstacles in the case of multiple stack symbols. This is
followed by some further examples from number theory in
Section~\ref{sec:numtheory} and a discussion in
Section~\ref{sec:ambiguity} on inherent ambiguity in indexed grammars.

\subsection{Notation}
Throughout, we use standard terminology with respect to formal
language~theory.  The expression $x\vert y$\ \ denotes ``$x$
exclusive-or $y$''.  We use epsilon {}``$\varepsilon$'' to denote the
empty word. The Kleene star operator applied to $x$, written $x^{*}$
means make zero or more copies of $x$. A related notation is $x^{+}$
which means make one or more copies of $x$. The word reversal of $w$
is indicated by $w^{R}$.  The length of the string $x$ is denoted
$\vert x\vert$. We print grammar variables in upper case bold
letters. Grammar terminals are in lower case italic. We use the symbol
$\overset{*}{\to}$ to indicate the composition of two or more grammar
productions.

\section{The case of one index symbol}
\label{sec:onesymb}
\subsection{Generalizing the DSV process}
A central objective of this work is to answer enumerative questions
about the number of words of a given length in a language defined by
an indexed grammar. It turns out that in many cases the classic
translation of production rules into functional equations satisfied by
generating functions works here. This type of strategy was first
applied to formal languages in~\cite{ChSc1963}, and the ideas have
been expanded to handle more general combinatorial
equations~\cite{FlSe2009}. We summarize the process below. The challenge
posed by indexed grammars is that an infinite number of equations are
produced by the process. Nonetheless, we identify some sufficiency criteria which allow
us to solve the system in some reasonable way.

First, let us recall the translation process for context-free
grammars. Let $\mathfrak{G}$ be a context-free
grammar specification for a non-empty language $\mathfrak{L}$, with
start symbol \textbf{$\mathbf{S}$}. For each production, replace each
terminal with the formal variable $z$ and each non-terminal $\bA$ with
a formal power series $A(z)$. Translate the grammar symbols
$\rightarrow,\ |,\ \varepsilon$ into $=,\ +,\ 1$, respectively, with
juxtaposition becoming commutative multiplication.  Thus the grammar
is transformed into a system of equations. We summarize the main
results of this area as follows.
\begin{thm*}[Chomsky-Sch\"utzenberger]
Each formal power series $A(z)=\sum A_{n}z^{n}$ in the above transformation
is an ordinary generating function where $A_{n}$ is an integer representing
the number of word productions of length $n$ that can be realized
from the non-terminal $\bA$. In particular, if the original
context-free grammar is unambiguous, then $S(z)$ is the growth series
for the language $\mathfrak{L}$, in which case $S(z)$ is an algebraic
function. 
\end{thm*}

A context-free grammar has only finitely many non-terminals. In the
case of an indexed grammar, we treat a single variable $\bA$ as
housing recursively many non-terminals, one for each distinct index
string carried by $\bA$ (although only finitely many are displayed in
parsing any given word). To generalize the DSV procedure to indexed
grammars we apply the same transformation scheme to the grammar,
under the viewpoint that every production rule is shorthand for an infinite
set of productions, where non-terminals are paired with index
strings. The generating functions are thus also similarly indexed, for
example $\bA_{gfghgf\$}$ gives rise to $A_{gfghgf\$}(z)$.

Initially, this is most unsatisfying, since the transformation recipe
produces a system of \emph{infinitely many equations in infinitely
 many functions}! We are unable to describe a general scheme to solve
these equations, and even in some small cases we do not obtain
satisfying expressions (see Example~\ref{ex:difficult_recursion} below). However, if there is only one index symbol
(disregarding the end of stack symbol), and some other conditions are
satisfied, we can outline a procedure to reduce the system. When there
is only one stack symbol, it is sufficient to identify only the
size of the stack in a non-terminal, stack pair. For example, in the
translation to functions $\bA_{ffffff\$}$ becomes $A_6(z)$.

\begin{example}
\label{ex:an2} The language 
$\mathfrak{L}_{sqr}=\left\{ a^{2^{n}}:\ n\geq0\right\} $
is generated by an indexed grammar. The enumeration for this
example is trivial but the example serves a pedagogical
purpose, as it sets up the process in the case of one index
symbol, and provides an example of an indexed language whose
generating function is not differentiably algebraic. 

As usual, we use \$ to indicate the bottom-most index symbol. Disregarding this,
there is only one index symbol actually used, and so in the
translation process we note only the size of the stack, not its
contents. Furthermore we identify $S_0(z)=S(z)$.

\begin{equation*}
\begin{array}{llll}
\mathbf{S}\rightarrow\mathbf{T}_{\$}&\mathbf{T}\rightarrow\mathbf{T}_{f}\vert\mathbf{D}&\mathbf{D}_{f}\rightarrow\mathbf{DD}&\mathbf{D}_{\$}\rightarrow
a\\[.25cm]
S_0(z)=T_0(z)& T_n(z)=T_{n+1}(z)+D_n(z)&
D_{n+1}(z)=D_n(z)^2& D_0(z)=z.\\[.25cm]
\end{array}
\end{equation*}

Observe that indices are loaded onto $\mathbf{T}$ then transferred to 
$\mathbf{D}$ which is then repeatedly doubled ($\mathbf{D}$ is a mnemonic
for \lq\lq duplicator\rq\rq ). After all doubling, each instance of 
$\mathbf{D}_\$ $ becomes an $a$.

Immediately we solve $D_n(z)=D_0(z)^{2^n}=z^{2^n}$, and the system of grammar 
equations becomes
\[
S_0(z)=T_{0}(z)=T_{1}(z)+D_{0}(z)=T_{2}+D_{1}+D_{0}=\cdots=\sum_{n\geq0}D_{n}(z)
=\sum_{n\geq0}z^{2^{n}}.  \]
We observe that the sequence of partial sums converge as a power
series inside the unit circle, and that the $T_n(z)$ are incrementally 
eliminated. We refer to this process, summarized in
Proposition~\ref{prop:push-to-inf} below, as \emph{pushing the $T_{n}(z)$
off to infinity}. 

The function $S(z)$ satisfies the functional equation
$S(z)=z+S(z^2)$. The series diverges at $z=1$, and hence $S(z)$ is
singular at $z=1$. However, by the functional equation it also
diverges at $z=-1$. By repeated application of this argument, we can
show that $S(z)$ is singular at every $2^n-th$ root of unity. Thus it
has an infinite number of singularities, and it cannot be D-finite.
In fact, $S(z)$ satisfies no algebraic differential equation
\cite{LipRub1986}. Consequently, the class of generating functions for
indexed languages is not contained in the class of differentiably
algebraic functions.

\end{example}
\subsection{A straightforward case: Balanced indexed grammars}
Next, we describe a condition that allows us to guarantee that this
process will result in a simplified generating function expression for $S(z)$.

\begin{defn}
An indexed grammar is \emph{balanced} provided there are constants
$C,K\geq0$, depending only on the grammar, such that the longest string
of indices associated to any non-terminal in any sentential form $\mathcal{W}$
has length at most $C\vert w\vert+K$ where $w$ is any terminal word
produced from $\mathcal{W}$. (Note: in all our balanced examples
we can take $C=1$ and $K\in\left\{ 0,1\right\}$.)\end{defn}

\begin{prop}
 \label{prop:push-to-inf} Let $\mathfrak{G}=(\cN,\cT,\cI,\cP,\bS)$ be
 an unambiguous, balanced, indexed grammar in strongly reduced form
 for some non-empty language $\mathfrak{L}$ with
 $\cI=\{f\}$. Furthermore, suppose that $\mathbf{V}\in\cN$ is the
 only non-terminal that loads $f$ and that the only allowable production 
 in which $\mathbf{V}$ appears on the right side is 
 $\mathbf{S}\to\mathbf{V}_\$ $. Then in the generalized DSV
 equations for $\mathfrak{G}$, the sequence of functions
 $V_n(z)\equiv V_{f^n}(z)$ can be eliminated (pushed to
 infinity). Under these hypotheses, the system of equations defining
 $S(z)$ reduces to finitely many bounded recurrences with initial
 conditions whose solution is the growth function for $\mathfrak{L}$.
\end{prop}
\begin{proof}
By hypothesis, there is only one load production and it has form
$\mathbf{V}\to\mathbf{V}_f$. Without loss of generality we may assume 
there is no type~\eqref{pop} rule $\mathbf{V}_f\to\beta$ (in fact, such a 
rule can be eliminated by the creation of a new variable $\mathbf{U}$ 
and adding new productions $\mathbf{V}\to\mathbf{U}$ and 
$\mathbf{U}_f\to\beta$ ). Necessarily there will be at least one context-free 
type rule $\mathbf{V}\to\beta$ (else $\mathbf{V}$ is a useless symbol).

Consider all productions in $\mathcal{\mathfrak{G}}$ that have $\mathbf{V}$
on the left side. Converting these productions into the usual functional
equations and solving for $V_n(z)$ gives an equation of form 
\[ V_{n}(z)=V_{n+1}(z)+W_{n\pm e}(z) \]
where $W_{n\pm e}(z)$ denotes an expression that represents all
other grammar productions having $\mathbf{V}$ on the left side and
$e\in\left\{ 0,1,-1\right\} $. 

We make the simplifying assumption that
$e=0$ for the remainder of this paragraph.
Starting with $n=0$ and iterating $N\gg0$ times yields $V_{0}(z)=V_{N}(z)+W_{0}(z)+W_{1}(z)+\cdots+W_{N}(z)$.
By the balanced hypothesis, there exists a constants $C,K\geq0$ such
that all terminal words produced from $\mathbf{V}_{f^{N}}$ have length
at least $N/C-K\gg0.$ This means that the first $N/C-K$ terms in the
ordinary generating function for $V_{0}(z)$ are unaffected by the 
contributions from $V_{N}(z)$
and depend only on the fixed sum $W_{0}(z)+W_{1}(z)+\cdots+W_{N}(z)$.
Therefore the $(N/C-K)^{th}$ partial sum defining the generating function for 
$V_{0}(z)$
is stabilized as soon as the iteration above reaches $V_{N}(z)$.
This is true for all big $N$, so we may take the limit as $N\rightarrow\infty$
and express $V_{0}(z)=\sum_{n\geq0}W_{n}(z)$ .

Allowing $e=\pm1$ in the previous paragraph
merely shifts indices in the sum and does not affect the logic
of the argument. Therefore, in all cases the variables $V_{n}(z)$ 
for each $n>0$ are
eliminated from the system of equations induced by the grammar
$\mathfrak{G}$. We assumed that $\mathbf{V}$ was the only variable
loading indices, so all other grammar variables either unload/pop
indices or are terminals. Consequently, the remaining functions
describe a finite triangular system, with finitely many finite
recurrences of bounded depth and known initial conditions.  The
solution of this simplified system is $S(z)$.
\end{proof}

We observe that the expression for $V_{0}(z)$ derived above actually converges
as an analytic function in a neighborhood of zero (see section~\ref{sec:ogf}
above). It turns out that the balanced hypothesis used above is 
already satisfied.

\begin{lem}
\label{lem:balanced1}
Suppose the indexed grammar $\mathfrak{G}$ is unambiguous, strongly reduced,
$\varepsilon$-free, and has only one index symbol $f$ (other than \$). Then 
$\mathfrak{G}$ is balanced.
\end{lem}

\begin{proof}
If the language produced by $\mathfrak{G}$ is finite, there is nothing to
prove so we may assume the language is infinite. Let us define a special 
sequence of grammar productions used in producing a terminal word. Suppose
a sentential form contains several variables, each of which is ready for
unloading of the index symbol $f$. A \textit{step} will consist of unloading
a single $f$ from each of these non-terminals, starting from the leftmost
variable. After the step, each of these
variables will hold an index string that is exactly one character shorter than 
before. 

Consider a sentential form $F$ containing one or more non-terminals 
in the unloading stage and each of whose index strings are 
of length at least $N\gg0$. These symbols can only by unloaded at the rate 
of one per production rule (this is the reduced hypothesis) and we'll
only consider changing $F$ by steps. 

On the other hand there, are only finitely many
production rules to do this unloading of finitely many variables. 
Thus for large $N$ there is a cycle of production rules as indices are unloaded,
with cycle length bounded by a global constant $C>0$ which depends only
on the grammar. Furthermore, this
cycle is reached after at most $K<C$ many productions. Let $F'$ 
denote the sentential form that results from $F$ as one such cycle is 
begun and let $F''$ be the sentential form after the cycle is applied to
$F'$. 

Consider lengths of sentential forms (counting terminals and variables but 
ignoring indices). Since the grammar is reduced and $\varepsilon$-free, 
each grammar
production is a non-decreasing function of lengths. Thus
$\vert F''\vert\geq\vert F'\vert\geq\vert F\vert$. Discounting any indices,
the equality of sentential forms $F''=F'$ is not possible because this implies 
ambiguity. 

We claim that either $F''$ is longer than $F'$ or that $F''$ has more terminals
than $F'$. If not, then $F''$ has exactly the same terminals as $F'$, and each 
has the same quantity of variables. There are only finitely many arrangements 
of terminals and variables for this length and for large $N$ we may loop
stepwise through the production cycle arbitrarily often and thus repeat exactly
a sentential form (discounting indices). This implies our grammar is ambiguous 
contrary to hypothesis. 

Thus after $C$ steps the sentential forms either obtain at least one new 
terminal or 
non-terminal. In the latter case, variables must convert into terminals on
\$ (via the reduced, unambiguous, $\varepsilon$-free hypotheses). There will 
be at least one terminal per step in the final output word $w$. We obtain the
inequality $(N-K)/C \leq\vert w\vert $ which establishes the lemma.
\end{proof}

\subsection{A collection of examples}
We illustrate the method, its power and its limitations with three
examples. The first two examples exhibit languages with intermediate 
growth and show that some of the hypotheses of 
Proposition~\ref{prop:push-to-inf} can be relaxed. We are unable to 
resolve the third example to our satisfaction.

The first example is originally due to~\cite{GrMa1999} and features
the indexed language \[\mathfrak{L}_{G/M}=\left\{
  ab^{i_{1}}ab^{i_{2}}\cdots ab^{i_{k}}\ :\ 0\leq i_{1}\leq
  i_{2}\leq\cdots\leq i_{k}\right\}\] with intermediate growth
(meaning that the number of words of length $n$ ultimately grows
faster than any polynomial in $n$ but more slowly than $2^{kn}$ for
any constant $k>0$). The question of whether a context-free language
could have this property was asked in~\cite{Flaj1987} and answered in
the negative~\cite{Inci2001,BrGi2002}. Grigorchuk and Mach\'{i}
constructed their language based on the generating function of Euler's
partition function. A word of length $n$ encodes a partition sum of
$n$. For instance, the partitions of $n=5$ are $1+1+1+1+1,\ 1+1+1+2,\
1+2+2,\ 1+1+3,\ 2+3,\ 1+4,\ 5$. The corresponding words in
$\mathfrak{L}_{G/M}$ are $aaaaa,\ aaaab,\ aabab,\ aaabbb,\ ababb,\
aabbb,\ abbbb$, respectively. The derivation below is ours.
\begin{example}
\label{ex:GM}
An unambiguous grammar for $\mathfrak{L}_{G/M}$ is\[
\mathbf{S}\rightarrow\mathbf{T}_{\$}\qquad\mathbf{T}\rightarrow\mathbf{T}_{f}\vert\mathbf{GT}\vert\mathbf{G}\qquad\mathbf{G}_{f}\rightarrow\mathbf{G}b\qquad\mathbf{G}_{\$}\rightarrow a\]
The latter two productions imply that 
$\mathbf{G}_{f^{m}\$}\overset{*}{\to}ab^{m}$
or in terms of functions $G_{m}(z)=z^{m+1}$. A typical parse tree is
illustrated in Figure~\ref{fig:part}.
\begin{figure}
\center
\begin{tikzpicture}[level/.style={sibling distance=60mm/#1,level distance=10mm}]
\node  (z){$\mathbf{S}$}
 child {node (U) {$\mathbf{T}_{f^{i_1}\$}$} edge from parent [->] 
    child  {node (V) {$\mathbf{G}_{f^{i_1}\$}$}
          child  {node (b) {$ab^{i_1}$} edge from parent node  [left=-2pt] {{\tiny$*$}}}
              }
    child  {node (R) {$\mathbf{T}_{f^{i_1}\$}$}
           child {node  (V2) {$\mathbf{T}_{f^{i_2}\$}$} 
                    child {node (V3) {$\mathbf{G}_{f^{i_2}\$}$} 
                           child {node  (w2) {$ab^{i_2}$} edge from
                             parent node  [left=-2pt] {{\tiny$*$}}  }
                           }
                    child {node (Ti2) {$\mathbf{T}_{f^{i_2}\$}$}
                              child {node (Ti3) {$\mathbf{T}_{f^{i_3}\$}$} 
                              child {node (Tik) {$ab^{i_3}$} edge from
                                         parent node [left=-2pt] {{\tiny$*$}}}
                              child {node (V4) {$\mathbf{T}_{f^{i_k}\$}$} edge from parent[dashed]
                                     child {node  (w4)
                                       {$ab^{i_k}$} edge from parent
                                       [solid] node  [left=-2pt] {{\tiny$*$}}}
                                     } edge from parent node  [left=-2pt] {{\tiny$*$}}
                                   }  
                     }edge from parent node [left=-2pt] {{\tiny$*$}}}
             } node  [left=-2pt] {{\tiny$*$}}
       };
\end{tikzpicture}

\caption{A typical parse tree in the grammar 
\[\mathbf{S}\rightarrow\mathbf{T}_{\$}\quad
 \mathbf{T}\rightarrow\mathbf{T}_{f}\vert\mathbf{GT}\vert\mathbf{G}\quad
 \mathbf{G}_{f}\rightarrow\mathbf{G}b\quad\mathbf{G}_{\$}\rightarrow
 a\]}
\label{fig:part}
\end{figure}

The second grammar production group transforms to
\[
T_{m}(z)=T_{m+1}(z)+G_{m}(z)T_{m}(z)+G_{m}(z)\ .\]
Substitution and solving for $T_{m}$ gives \[
T_{m}(z)=\frac{z^{m+1}+T_{m+1}(z)}{1-z^{m+1}}\ .\]
Iterating this recurrence yields a kind of inverted continued fraction:\[
S(z)=T_{0}(z)=\frac{z+T_{1}(z)}{1-z}=\frac{z+\frac{z^{2}+T_{2}(z)}{1-z^{2}}}{1-z}=\frac{z+\frac{z^{2}+\frac{z^{3}+\
     \Ddots}{1-z^{3}}}{1-z^{2}}}{1-z}\:.\]
Equivalently, this recurrence can be represented as
\begin{align*}
\frac{z+T_{1}(z)}{1-z}&=\frac{z(1-z^{2})+z^{2}+T_{2}(z)}{(1-z)(1-z^{2})}=\frac{z(1-z^{2})(1-z^{3})+z^{2}(1-z^{3})+z^{3}+T_{3}(z)}{(1-z)(1-z^{2})(1-z^{3})}\\
\intertext{or}
S(z)&=\frac{z}{1-z}+\frac{z^{2}}{(1-z)(1-z^{2})}+\frac{z^{3}}{(1-z)(1-z^{2})(1-z^{3})}+\cdots+\frac{z^{k}+T_{k}(z)}{\prod_{n=1}^{k}(1-z^{n})}\:.
\end{align*}
Even though this grammar allows the index loading variable $\mathbf{T}$ to 
appear
on the right side of the production $\mathbf{T}\to\mathbf{GT}$ (contrary to
one of the hypotheses in Proposition~\ref{prop:push-to-inf}) the 
convergence proof of Proposition~\ref{prop:push-to-inf} is applicable to 
the expression above allowing us to push $T_{k}(z)$ off to infinity:
\[
S(z)=\sum_{j\geq1}\frac{z^{j}}{(1-z)(1-z^{2})\cdots(1-z^{j})}\ .\]
Here we recognize a classic combinatorial summation of partitions
in term of their largest part~\cite[Example I.7]{FlSe2009}. Thus,
we have recovered the ordinary generating function for partitions,
$S(z)=\sum_{n\geq1}p(n)z^{n}$, where the coefficients belong to Euler's
partition sequence $p(n)$. Since we can also write \[
\sum_{n\geq1}p(n)z^{n}=\prod_{n\geq1}\frac{1}{1-z^{n}}\ ,\]
it is true that $S(z)$ has a dense set of singularities on the unit
circle and is not D-finite. 
\end{example}

In general, allowing an index loading variable $\mathbf{V}$ to appear on 
both sides of a context-free rule implies that the
corresponding DSV system of equations will have an algebraic (but
not necessarily linear) equation expressing $V_n(Z)$ in terms of $z$
and $V_{n+1}(z)$. However, multiple occurrences of $\mathbf{V}$ on the right 
side of such a production yield an ambiguous grammar with meaningless
$S(z)$ in the corresponding DSV reduction! In fact, suppose
$$\mathbf{S}\to\mathbf{V}_\$ \qquad \mathbf{V}\to \mathbf{V}_f 
\vert \alpha\mathbf{V}\beta\mathbf{V}\gamma $$
comprise part of a reduced indexed grammar where 
$\alpha, \beta, \gamma\in ( \cT\cup\cN\cup\varepsilon )^* $. Then 
the distinct production chains
$$\mathbf{S} \overset{*}{\to} \mathbf{V}_{ff\$} \to
\alpha\mathbf{V}_{ff\$}\beta\mathbf{V}_{ff\$}\gamma \eqno(+)$$
and
$$\mathbf{S} \overset{*}{\to} \mathbf{V}_{f\$}\to
\alpha\mathbf{V}_{f\$}\beta\mathbf{V}_{f\$}\gamma\to
\alpha\mathbf{V}_{ff\$}\beta\mathbf{V}_{f\$}\gamma\to
\alpha\mathbf{V}_{ff\$}\beta\mathbf{V}_{ff\$}\gamma 
\eqno(++)$$
yield identical output streams. 

On the other hand, Example~\ref{ex:GM} shows that a single instance of
the loading symbol can appear on the right side of a context-free rule
without leading to ambiguity of the underlying grammar. It is not hard
to see that the resulting DSV equation is always linear in $V_n(z)$. We have
proved

\begin{prop}
 \label{prop:push-to-inf II} Let $\mathfrak{G}=(\cN,\cT,\cI,\cP,\bS)$ be
 an unambiguous, balanced (or $\varepsilon$-free), indexed grammar in strongly
 reduced form
 for some non-empty language $\mathfrak{L}$ with
 $\cI=\{f\}$. Suppose that $\mathbf{V}\in\cN$ is the
 only non-terminal that loads $f$. Then the DSV
 system of equations defining
 $S(z)$ reduces to finitely many bounded recurrences with initial
 conditions whose solution is the growth function for $\mathfrak{L}$.
\end{prop}

\begin{example}
\label{ex:ig} Another series with intermediate growth can be realized
as the ordinary generating function of the following indexed grammar:
\begin{eqnarray*}
&\mathbf{S}\rightarrow\mathbf{C}\vert\mathbf{C}\mathbf{T}_{\$}\qquad\mathbf{C}\rightarrow b\mathbf{C}\vert\varepsilon\qquad\mathbf{T}\rightarrow\mathbf{T}_{f}\vert\mathbf{W}\qquad\mathbf{W}_{f}\rightarrow\mathbf{VWX}\\
&\mathbf{V}_{f}\rightarrow
aa\mathbf{V}\qquad\mathbf{V}_{\$}\rightarrow
aa\qquad\mathbf{W}_{\$}\rightarrow a\qquad\mathbf{X}\rightarrow a\vert
b
\end{eqnarray*}
As usual, we use index \$ to indicate the bottom of the stack, with
$f$ being the only actual index symbol. (The reader may notice that the rule
$\mathbf{S}\to\mathbf{CT_\$} $ is yet another relaxation of the hypotheses of
Proposition~\ref{prop:push-to-inf} that does not affect its conclusion.)
A typical parse tree is given
in Figure~\ref{fig:intermed}. From this we see that the language generated (unambiguously) is \[
\mathfrak{L}_{int}=\left\{ b^{*}\left(\varepsilon\ \vert\
   a^{n^{2}+n+1}\left(a\vert b\right)^{n}\right)\ :\
 n\geqslant0\right\} .\]

\begin{figure}
\center
\begin{tikzpicture}[
level 1/.style={sibling distance=30mm,level distance=10mm},
level 3/.style={sibling distance=25mm,level distance=10mm},
level 4/.style={sibling distance=45mm,level distance=12mm},
level 5/.style={sibling distance=25mm},
level 6/.style={sibling distance=15mm}
]
\node (a) {$\mathbf{S}$} {
 child {node (b) {$\mathbf{C}$} edge from parent [->] 
            child {node (c) {$b^*$} edge from parent node [left=0.5pt] {\tiny$*$}}}
 child {node (d) {$\mathbf{T}_{\$}$} edge from parent [->] 
           child  {node (e) {$\mathbf{T}_{f^n\$}$} 
                      child  {node (f) {$\mathbf{W}_{f^n\$}$}
                                 child {node  (g) {$\mathbf{V}_{f^{n-1}\$}$} 
                                          child {node (h) {$a^{2n}$} edge from
                                                   parent node  [left=-2pt] {{\tiny$*$}} }
                                          }
                                 child {node  (i) {$\mathbf{W}_{f^{n-1}\$}$} 
                                          child {node  (j) {$\mathbf{V}_{f^{n-2}\$}$} 
                                                    child {node (k) {$a^{2(n-1)}$} edge from
                                                      parent node  [left=-2pt] {{\tiny$*$}} }
                                                     }
                                          child {node  (l)  {$\mathbf{W}_{f^{n-2}\$}$}
                                                   child {node (l1) {} edge from parent [dashed]}
                                                   child {node (l2) {$\mathbf{W}_{\$}$} edge from parent [dashed]
                                                              child{node (l4) {$a$} edge from parent [solid]}}
                                                   child {node (l3) {} edge from parent [dashed]}
                                                 }
                                         child {node  (m) {$\mathbf{X}_{f^{n-2}\$}$} 
                                                   child {node (n) {$a|b$} edge from
                                                   parent node [left=-2pt] {{\tiny$*$}} }
                                                   }
                                         }
                                 child {node  (o) {$\mathbf{X}_{f^{n-1}\$}$} 
                                          child {node (m) {$a|b$} edge from
                                                   parent node  [left=-2pt] {{\tiny$*$}} }
                                          }
                           }
               edge from parent node [left=-2pt] {{\tiny$*$}} }
            }
  };
\end{tikzpicture}
\caption{A typical parse tree in the grammar
\[
\mathbf{S}\rightarrow\mathbf{C}\vert\mathbf{C}\mathbf{T}_{\$}\qquad\mathbf{C}\rightarrow b\mathbf{C}\vert\varepsilon\qquad\mathbf{T}\rightarrow\mathbf{T}_{f}\vert\mathbf{W}\qquad\mathbf{W}_{f}\rightarrow\mathbf{VWX}\]
\[
\mathbf{V}_{f}\rightarrow aa\mathbf{V}\qquad\mathbf{V}_{\$}\rightarrow aa\qquad\mathbf{W}_{\$}\rightarrow a\qquad\mathbf{X}\rightarrow a\vert b\]
}
\label{fig:intermed}
\end{figure}

We can derive the generating function in the usual fashion. Note the
shortcuts $\mathbf{X}_{f^{n}\$}\rightarrow(a\vert b)$ (regardless
of indices) and $\mathbf{V}_{f^{n}\$}\overset{*}{\rightarrow}a^{2n}\mathbf{V}_{\$}\rightarrow a^{2n+2}$.
Starting with \[
\mathbf{W}_{ff\$}\rightarrow\mathbf{V}_{f\$}\mathbf{W}_{f\$}\mathbf{X}_{f\$}\overset{*}{\rightarrow}a^{4}\mathbf{W}_{f\$}(a\vert b)\rightarrow a^{4}\mathbf{V}_{\$}\mathbf{W}_{\$}\mathbf{X}_{\$}(a\vert b)\overset{*}{\rightarrow}a^{4}a^{2}a(a\vert b)^{2}\]
one can use induction to derive
\[
\mathbf{W}_{f^{n}\$}\overset{*}{\rightarrow}a^{2n}\cdots a^{4}a^{2}a(a\vert b)^{n}=a^{n(n+1)}a(a\vert b)^{n}=a^{n^{2}+n+1}(a\vert b)^{n}.\]
In terms of generating functions these shortcuts imply $W_{n}(z)=z^{n^{2}+n+1}2^{n}z^{n}=2^{n}z^{(n+1)^{2}}$;
also $C(z)=\frac{1}{1-z}$. Put this all together to get
\begin{align*}
S(z)&=C(z)+C(z)T_{0}(z)=C\left(1+T_{0}\right)=C\left(1+T_{1}+W_{0}\right)=\\
&=C(1+T_{2}+W_{0}+W_{1})=\cdots=C\left(1+\sum_{n=0}^{\infty}W_{n}\right)\\
&=\frac{1}{1-z}\left(1+\sum_{n=0}^{\infty}2^{n}z^{(n+1)^{2}}\right).
\end{align*}
Write as a sum of rational functions and expand each geometric series:
\[
\begin{array}{rcl@{}l@{}l@{}l@{}l@{}l@{}l@{}l@{}l@{}}
S(z) & = & \multicolumn{7}{l}{\frac{1}{1-z}+\frac{z}{1-z}+\frac{2z^{4}}{1-z}+\frac{4z^{9}}{1-z}+\frac{8z^{16}}{1-z}+\cdots}\\[3mm]
& = & 1 & +z+z^{2}+z^{3} & +z^{4} & +z^{5} & +z^{6} & +z^{7} & +z^{8} & +z^{9} & +\cdots\\
&  &  & +z+z^{2}+z^{3} & +z^{4} & +z^{5} & +z^{6} & +z^{7} & +z^{8} & +z^{9} & +\cdots\\
&  &  &  & +2z^{4} & +2z^{5} & +2z^{6} & +2z^{7} & +2z^{8} & +2z^{9} & +\cdots\\
&  &  &  &  &  &  &  &  & +4z^{9} & +\dots\\
&  &  &  &  &  &  &  &  & \qquad\ddots\end{array}\]
and so forth. Sum the columns and observe that the coefficient of
each $z^{n}$ is a power of $2$, with new increments occurring when
$n$ is a perfect square. Thus \[
S(z)={\displaystyle \sum_{n=0}^{\infty}}2^{\left\lfloor \sqrt{n}\right\rfloor }z^{n}\]
and the coefficient of $z^{n}$ grows faster than any polynomial (as
$n\rightarrow\infty$) but is sub-exponential. 
\end{example}
The indexed grammars used in applications (combings of groups, combinatorial
descriptions, etc) tend to be reasonably simple and most use one index
symbol. Despite the success of our many examples, Propositions~\ref{prop:push-to-inf} /~\ref{prop:push-to-inf II} 
do \emph{not}
guarantee that an explicit closed formula for $S(z)$ can always be
found. 
\begin{example}\label{ex:difficult_recursion}
Consider the following balanced grammar:\[
\mathbf{S}\rightarrow\mathbf{T}_{\$}\qquad\mathbf{T}\rightarrow\mathbf{T}_{f}\vert\mathbf{N}\qquad\mathbf{N}_{f}\rightarrow a\mathbf{N}\vert b^{2}\mathbf{M}\qquad\mathbf{M}_{f}\to ab\mathbf{NM}\qquad\mathbf{M}_{\$},\mathbf{N}_{\$}\rightarrow\varepsilon\]
The hypotheses of Proposition~\ref{prop:push-to-inf II} are satisfied so we
can push $T_n$ to infinity and obtain

\[
S(z)=T_{0}=N_{0}+T_{1}=N_{0}+N_{1}+T_{2}=\cdots=\sum_{n\geq0}N_{n}(z)\ .\]
However, the recursions defining $N_{n}(z)$ are intertwined and formidable:
\[
N_{n}(z)=zN_{n-1}+z^{2}M_{n-1}\qquad\mathrm{and}\qquad M_{n}(z)=z^{2}N_{n-1}M_{n-1}\quad\forall n\geq1\]
with $N_{0}=1=M_{0}$. It is possible to eliminate $M$ but the resulting
nonlinear recursion\[
N_{n}(z)=zN_{n-1}+z^{2}N_{n-1}N_{n-2}-z^{3}N_{n-2}^{2}\]
does not appear to be a bargain (it is possible that a multivariate generating
function as per Example~\ref{ex:BG} may be helpful). 
\end{example}

\section{The case of several index symbols}
\label{sec:multiple}
Multiple index symbols increase the expressive power, and under
certain conditions, and by grouping stacks into equivalence classes we can apply a similar technique.

Our next example uses two index symbols (in addition to \$) in an
essential way.

\begin{example}
\label{ex:serial} Define $\mathcal{\mathfrak{L}}_{serial}=\left\{ \left(ab^{i}c^{j}\right)^{+}\ :\ 1\leq i\leq j\right\} $.
Consider the grammar:
\begin{eqnarray*} 
&\mathbf{S}\rightarrow\mathbf{T}_{\$}\qquad\mathbf{T}\rightarrow\mathbf{T}_{g}\vert\mathbf{U}_{f}\qquad\mathbf{U}\rightarrow\mathbf{U}_{f}\vert\mathbf{VR}\vert\mathbf{V}\qquad\mathbf{R}\rightarrow\mathbf{VR}\vert\mathbf{V}\\
&\mathbf{V}\rightarrow
a\mathbf{BC}\qquad\mathbf{B}_{f}\rightarrow\mathbf{B}b\qquad\mathbf{B}_{g}\rightarrow\varepsilon\qquad\mathbf{C}_{f}\rightarrow\mathbf{C}c\qquad\mathbf{C}_{g}\rightarrow
c\mathbf{C}\qquad\mathbf{C}_{\$}\rightarrow\varepsilon
\end{eqnarray*}
Observe that the two index symbols are loaded serially: all $g$'s
are loaded prior to any $f$ so each valid index string will be of
the form $f^{+}g^{*}\$$. We also have the shortcuts $\mathbf{C}_{f^{m}g^{n}\$}\overset{*}{\to}c^{m}\mathbf{C}_{g^{n}\$}\overset{*}{\to}c^{m+n}$
and $\mathbf{B}_{f^{m}g^{n}\$}\overset{*}{\to}b^{m}\mathbf{B}_{g^{n}\$}\rightarrow b^{m}\varepsilon=b^{m}$
and consequently $\mathbf{V}_{f^{m}g^{n}\$}\rightarrow
a\mathbf{B}_{f^{m}g^{n}\$}\mathbf{C}_{f^{m}g^{n}\$}\overset{*}{\to}ab^{m}c^{m+n}$. A
typical parse tree is given in Figure~\ref{fig:serial}. 
\begin{figure}
\center
\begin{tikzpicture}[level/.style={sibling distance=80mm/#1},level distance=10mm]
\node  (z){$\mathbf{S}$}
 child {node (U) {$\mathbf{U}_{f^mg^n\$}$} edge from parent [->] 
    child  {node (V) {$\mathbf{V}_{f^mg^n\$}$}
          child  {node (b) {$ab^nc^{n+m}$} edge from parent node  [left=-2pt] {{\tiny$*$}}}
              }
    child  {node (R) {$\mathbf{R}_{f^mg^n\$}$}
           child {node  (V2) {$\mathbf{V}_{f^mg^n\$}$}
                   child {node  (w1) {$ab^nc^{n+m}$} edge from parent node  [left=-2pt] {{\tiny$*$}}}
                   }
           child {node  (R2) {$\mathbf{R}_{f^mg^n\$}$} edge from parent[dashed]
                    child {node (V3) {$\mathbf{V}_{f^mg^n\$}$} edge from parent[solid]
                           child {node  (w2) {$ab^nc^{n+m}$} edge
                             from parent node [left=-2pt] {{\tiny$*$}}}
                           }
                    child {node (R3) {$\mathbf{R}_{f^mg^n\$}$} edge from parent[solid]
                           child {node (V3) {$\mathbf{V}_{f^mg^n\$}$} 
                                     child {node  (w3)
                                       {$ab^nc^{n+m}$} edge from parent node [left=-2pt] {{\tiny$*$}}}
                                     }  
                          }  
                     } 
             } node [left=-2pt] {{\tiny$*$}}
       };
\end{tikzpicture}
\caption{A typical parse tree in the grammar \[ 
\mathbf{S}\rightarrow\mathbf{T}_{\$}\quad\mathbf{T}\rightarrow\mathbf{T}_{g}\vert\mathbf{U}_{f}\quad\mathbf{U}\rightarrow\mathbf{U}_{f}\vert\mathbf{VR}\vert\mathbf{V}\quad\mathbf{R}\rightarrow\mathbf{VR}\vert\mathbf{V}\]
\[\mathbf{V}\rightarrow
a\mathbf{BC}\quad\mathbf{B}_{f}\rightarrow\mathbf{B}b\quad\mathbf{B}_{g}\rightarrow\varepsilon\quad\mathbf{C}_{f}\rightarrow\mathbf{C}c\quad\mathbf{C}_{g}\rightarrow
c\mathbf{C}\quad\mathbf{C}_{\$}\rightarrow\varepsilon
\]
}
\label{fig:serial}
\end{figure}

The special form of the index strings ensures that such a string is
uniquely identified solely by the number of $f$'s and number of $g$'s
it carries. Consequently, the induced function $V_{f^{m}g^{n}\$}(z)$
can be relabelled more simply as $V_{m,n}(z)$, and similarly with functions
$T,U,R$. Working in the reverse order of the listed grammar productions,
we have the identities

\[
V_{m,n}(z)=z^{2m+n+1}\quad\mathrm{and}\quad R_{m,n}=\frac{V_{m,n}}{1-V_{m,n}}=\frac{z^{2m+n+1}}{1-z^{2m+n+1}}\ .\]
The grammar production $\mathbf{U}\rightarrow\mathbf{U}_{f}\vert\mathbf{VR}\vert\mathbf{V}$
implies for fixed $n>0$ that

\[
U_{1,n}(z)=U_{2,n}+\left(V_{1,n}R_{1,n}+V_{1,n}\right)=U_{3,n}+\left(V_{2,n}R_{2,n}+V_{2,n}\right)+\left(V_{1,n}R_{1,n}+V_{1,n}\right).\]
The hypothesis of Proposition~\ref{prop:push-to-inf II} are satisfied in
that we are dealing with a balanced grammar where currently only one
index symbol is being loaded onto one variable. Therefore for fixed
$n$ we can push $U_{m,n}(z)$ off to infinity and obtain

\[
U_{1,n}(z)=\sum_{m\geq1}\left(V_{m,n}R_{m,n}+V_{m,n}\right)=\sum_{m\geq1}R_{m,n}(z)=\sum_{m\geq1}\frac{z^{2m+n+1}}{1-z^{2m+n+1}}\ .\]
Our general derivation proceeds as follows:
\[
S(z)=T_{0,0}=T_{0,1}+U_{1,0}=T_{0,2}+U_{1,1}+U_{1,0}=\cdots=T_{0,k+1}+\sum_{n=1}^{k}U_{1,n}\]
Proposition~\ref{prop:push-to-inf II} can be invoked again to eliminate $T$.
We find that\[
S(z)=\sum_{n\geq1}\sum_{m\geq1}\frac{z^{2m+n+1}}{1-z^{2m+n+1}}=\sum_{j\geq1}\frac{z^{3j}}{(1-z^{j})(1-z^{2j})}=\sum_{i\geq1}\frac{z^{3i}+z^{4i}}{(1-z^{2i})^{2}}\]
with the latter two summations realized by expanding geometric series
and/or changing the order of summation in the double sum. In any event,
$S(z)$ has infinitely many singularities on the unit circle and is
not D-finite. 
\end{example}

\subsection{``Encode first, then copy''}\label{sec:encodefirst}
It is worth noting the reason why the copying schema used 
above succeeds in indexed grammars but fails for context-free
grammars.
The word $ab^{i}c^{j}$ is first encoded as an index string attached
to $\mathbf{V}$ and only then copied to $\mathbf{VR}$ (the grammar
symbol $\mathbf{R}$ is a mnemonic for {}``replicator''). This ensures
that $ab^{i}c^{j}$ is faithfully copied. Slogan: {}``encode first, then
copy''. Context-free grammars are limited to {}``copy first, then
express'' which does not allow for fidelity in copying.

We would like to generalize the previous example. The key notion was
the manner in which the indices were loaded. 
\begin{defn}
Suppose that $\mathfrak{G}$ is an unambiguous indexed grammar with
index alphabet $I=\left\{ f_{1},f_{2},\dots,f_{n}\right\} $ such
that every every index string $\sigma$ has form $f_{n}^{*}f_{n-1}^{*}\cdots f_{1}^{*}\$$.
We say the indices are \emph{loaded serially} in $\mathfrak{G}$.\end{defn}
\begin{cor}
Assume $\mathfrak{G}$ is an unambiguous balanced indexed grammar
with variable set $V$, non-terminal alphabet $\cT$, and index
alphabet $I=\left\{ f_{1},\dots,f_{n}\right\} $. Suppose all indices
are loaded serially onto respective variables $\mathbf{T}^{\left\{ 1\right\} },\dots,\mathbf{T}^{\left\{ n\right\} }$
and in the indicated order. Then each function family $T^{\left\{ j\right\} }(z)$
can be eliminated ({}``pushed to infinity'') and the system of equations
defining $S(z)$ can be reduced to finitely many recursions as per
the conclusion of Proposition~\ref{prop:push-to-inf}.\end{cor}
\begin{proof}
We have assumed that $\mathbf{T}^{\left\{ n\right\} }$ is the last
variable to load indices and is loaded with $f_{n}$ only, so $\mathbf{T}^{\left\{ n\right\} }$
carries an index string $\sigma$ of form $f_{n}^{*}\cdots f_{2}^{*}f_{1}^{*}\$$.
Indeed, any grammar production having $\mathbf{T}^{\left\{ n\right\} }$
on the left side will have a right side of two types: a string $\mathcal{U}_1\in\left(V\vert\cT\right)^{*}$
that loads $f_{n}$ onto $\mathbf{T}^{\left\{ n\right\} }$or a string
$\mathcal{U}_2\in\left(V\vert\cT\right)^{*}$without any loading
of indices. Neither of these types will include any variable $\mathbf{T}^{\left\{ j\right\} }$
for $j<n$ (by the serial loading hypothesis) nor will there be any
unloading of indices (by the unambiguous hypothesis). Consequently,
the equations having $T_{k}^{\left\{ n\right\} }(z)$ on the left
side have on their right side products and sums involving no $T^{\left\{ j\right\} }(z)$
for $j<n$ but only $T_{k}^{\left\{ n\right\} }(z)$ and functions
that define finite recurrences. The hypotheses of Proposition~\ref{prop:push-to-inf}
apply to this situation and $T^{\left\{ n\right\} }$ can be pushed
to infinity.

The previous paragraph is both the basis step and induction step of
an obvious argument that eliminates $T^{\left\{ n-1\right\} }$ then
$T^{\left\{ n-2\right\} }$ and so on till $T^{\left\{ 1\right\} }$
.
\end{proof}
In the case of a grammar with multiple index symbols, we would like
to be able to replace an unwieldy expression like $A_{gfghgf\$}(z)$
with $A_{2,3,1}(z)$ where the subscripts indicate two occurrences
of $f$, three of $g$, and one $h$. This is certainly possible for
a grammar with only one index symbol (excluding the end of stack marker
\$) or several symbols loaded serially as per the previous Corollary,
but is not possible in general. 
\begin{example}
(Ordering matters) Consider the language $\mathcal{\mathfrak{L}}_{ord}$
generated by the indexed grammar below.

\[
\mathbf{S}\rightarrow\mathbf{T}_{\$}\qquad\mathbf{T}\rightarrow\mathbf{T}_{\alpha}\vert\mathbf{T}_{\beta}\vert\mathbf{N}\qquad\mathbf{N}_{\alpha}\rightarrow a\mathbf{N}\qquad\mathbf{N}_{\beta}\rightarrow b\mathbf{N}b\mathbf{N}b\qquad\mathbf{N}_{\$}\rightarrow\varepsilon\]
When applying the DSV transformations to this grammar
we would like to write $N_{1,1}(z)$ as the formal power series corresponding
to the grammar variable $\mathbf{N}$ with any index string having
one $\alpha$ index and one $\beta$ index, followed by the end of
stack marker \$. Note the derivations $\mathbf{S}\overset{*}{\to}\mathbf{N}_{\alpha\beta\$}\overset{*}{\to}abbb$
and $\mathbf{S}\overset{*}{\to}\mathbf{N}_{\beta\alpha\$}\overset{*}{\to}babab$.
Even though both intermediate sentential forms have one of each index
symbol, followed by the end of stack marker \$, they produce distinct
words of differing length. Thus using subscripts to indicate
the quantity of stack indices cannot work in general without some
consideration of index ordering.

We note that the grammar is reduced and balanced. It is also unambiguous,
which can be verified by induction. In fact, if $\sigma\in\left(\alpha\vert\beta\right)^{*}\$$
is an index string such that $\mathbf{N}_{\sigma}\overset{*}{\to}w$
where $w$ is a terminal word of length $n$, then $\mathbf{N}_{\alpha\sigma}\overset{*}{\to}aw$
and $\mathbf{N}_{\beta\sigma}\overset{*}{\to}bwbwb$ where
$\vert aw\vert=n+1$ and $\vert bwbwb\vert=2n+3$. Suppose that all
words $w\in\mathcal{\mathfrak{L}}_{ord}$ of length $n$ or less are
produced unambiguously. Consider a word $v$ of length $n+1$. Either
$v=aw$ or $v=bw'bw'b$ for some shorter words $w,w'\in\mathfrak{L}_{ord}$
that were produced unambiguously by hypothesis. Clearly neither of
these forms for $v$ can be confused since one starts with $a$ and
the other with $b$.

The proof of Proposition~\ref{prop:push-to-inf} can be applied
to eliminate the $T_{\sigma}(z)$. Solving for $S(z)$ via the generalized
DSV procedure gives 
\[ S(z)=\sum_{\sigma\in I} N_{\sigma}(z) \]
where the sum is over all index strings $\sigma$. It is unfeasible to
simplify further because
the number of grammar functions $N_{\sigma}(z)$ grows exponentially
in the length of $\sigma$ without suitable simplifying recursions. 
\end{example}
The previous example showed two non-terminals having index strings
with the same quantity of respective symbols but in different orders
leading to two distinct functions. We can define a condition
that ensures such functions are the same. 
\begin{defn*}
Let $\mathcal{A}=\left\{ \alpha_{1},\alpha_{2},\dots,\alpha_{n}\right\} $
denote a finite alphabet. The \emph{Parikh vector} associated to $\sigma\in\mathcal{A}$
records the number of occurrences of each $\alpha_{i}$ in $\sigma$
as $x_{i}$ in the vector $[x_{1,}x_{2,}\dots x_{n}]$ (see~\cite{Pari1966}).
Define two strings $\sigma,\tau\in\mathcal{A}^*$ to be \textit{Parikh
equivalent} if they map to the same Parikh vector (in other words
$\tau$ is a permutation of $\sigma$). When $\mathcal{A}$ is the
index alphabet for a grammar, we extend this idea to non-terminals
and say $\mathbf{V}_{\sigma}$ is Parikh equivalent to $\mathbf{V}_{\tau}$
if $\sigma$ and $\tau$ map to the same Parikh vector. 
\end{defn*}
The following lemma gives sufficient conditions that
allow simplifying the ordering difficulty for function subscripts.
Its proof is immediate. 
\begin{lem}
\label{thm:Parikh} Assume $\mathfrak{G}$ is a nontrivial balanced
indexed grammar. Suppose each pair of Parikh equivalent index strings
$\sigma,\tau$ appended to a given grammar variable $\mathbf{V}$
result in identical induced functions $V_{\sigma}(z)\equiv V_{\tau}(z)$.
Then the functions induced from $\mathbf{V}$ can be consolidated
into equivalence classes (where we replace index string subscripts
by their respective Parikh vectors) without changing the solution
$S(z)$ of the system of DSV equations. 
\end{lem}
We have already used this lemma in Example~\ref{ex:serial} above.
We illustrate with another example. 
\begin{example}
\label{ex:ww} Consider the non-context-free language $\mathfrak{L}_{double}=\left\{ ww\ :\ w\in(a\vert b)^{*}\right\} $
produced by the grammar
\[
\mathbf{S}\to\mathbf{T}_{\$}\qquad\mathbf{T}\to\mathbf{T}_{\alpha}\vert\mathbf{T}_{\beta}\vert\mathbf{RR}\qquad\mathbf{R}_{\alpha}\to a\mathbf{R}\qquad\mathbf{R}_{\beta}\to b\mathbf{R}\qquad\mathbf{R}_{\$}\to\varepsilon\ .\]
Suppose $\sigma,\tau\in(\alpha\vert\beta)^{*}\$$ are Parikh equivalent
index strings of length $n$. It is clear that $R_{\sigma}(z)=z^{n}=R_{\tau}(z)$.
In fact, every string $u\in\left(\alpha\vert\beta\right)^{n}\$$ implies
$R_{u}(z)=z^{n}$, regardless of the particular distribution of $\alpha$
and $\beta$ in $u$. Instead of using the equivalence classes $R_{i,j}(z)$
where $[i,j]$ is the Parikh vector for $u$, let $R_{n}(z)$ denote
the equivalence class of all such induced functions $R_{u}(z)$
where $u\in(\alpha\vert\beta)^{n}\$$, and define $T_{n}(z)$ similarly.
We will abuse notation and refer to the elements of these classes
as $R_{n}(z)$ or $T_{n}(z)$, respectively. The grammar equations
become

\[
S(z)=T_{0}=2T_{1}+R_{0}^{2}=R_{0}^{2}+2\left(2T_{2}+R_{1}^{2}\right)=R_{0}^{2}+2R_{1}^{2}+4R_{2}^{2}+\cdots\]
where we can push the $T_{n}(z)$ to infinity as per the proof of
Proposition~\ref{prop:push-to-inf II}. Therefore \[
S(z)=\sum_{n\geq0}2^{n}R_{n}^{2}(z)=\sum_{n\geq0}2^{n}z^{2n}=\frac{1}{1-2z^{2}}\ .\]
\end{example}

\section{Further examples related to number theory}
\label{sec:numtheory}
In addition to our example from~\cite{GrMa1999} we have the following.
\begin{example}\label{ex:divisors}
 Define $\mathfrak{L}_{div}=\left\{ a^{n}\left(b^{n}\right)^{*}\ :\
   n>0\right\} $ which is generated by the unambiguous balanced
 grammar\footnote{As written, this grammar is not reduced because the
   rule $\bT\to\bA_f\mathbf{R}_f$ loads two indices
   simultaneously. However, by replacing that production by the pair
   $\bT\to\mathbf{U}_f\ ,\ \mathbf{U}_f\to \bA\mathbf{R}$ 
we obtain an equivalent grammar in
   reduced form. We use the former rule for brevity.}
\begin{eqnarray*}
&\mathbf{S}\to\mathbf{T}_{\$}\qquad\mathbf{T}\to\mathbf{T}_{f}\vert\mathbf{A}_f\mathbf{R}_f\qquad\mathbf{R}\to\mathbf{BR}\vert\varepsilon\\
&\mathbf{A}_{f}\to
a\mathbf{A}\qquad\mathbf{A}_{\$}\to\varepsilon\qquad\mathbf{B}_{f}\to
b\mathbf{B}\qquad\mathbf{B}_{\$}\to\varepsilon
\end{eqnarray*}

We see some familiar shortcuts: $\mathbf{A}_{f^{n}\$}\to a^{n}$ and
$\mathbf{B}_{f^{n}\$}\to b^{n}$. In terms of functions this means
$A_{n}(z)=z^{n}=B_{n}(z)$ and furthermore $R_{n}=B_{n}R_{n}+1$ implies
$R_{n}(z)=\frac{1}{1-z^{n}}$. Thus our main derivation becomes
\[
S(z)=T_{0}=T_{1}+A_{1}R_1=T_{2}+A_1R_1+A_2R_2=\cdots=\sum_{n\geq1}A_nR_n=\sum_{n\geq1}\frac{z^{n}}{1-z^{n}}\ .\]
Expand each rational summand into a geometric series and collect
terms \[
\begin{array}{rcl@{ }l@{}l@{}l@{}l@{}l@{}l@{}l@{}l@{}l@{}l@{}}
S(z) & = & \multicolumn{8}{l}{\frac{z}{1-z}+\frac{z^{2}}{1-z^{2}}+\frac{z^{3}}{1-z^{3}}+\frac{z^{4}}{1-z^{4}}+\cdots}\\[3mm]
& = & z & +z^{2} & +z^{3} & +z^{4} & +z^{5} & +z^{6} & +z^{7} & +z^{8} & +z^{9} & +z^{10} & +\cdots\\
&  &  & +z^{2} &  & +z^{4} &  & +z^{6} &  & +z^{8} &  & +z^{10} & +\cdots\\
&  &  &  & +z^{3} &  &  & +z^{6} &  &  & +z^{9} &  & +\cdots\\
&  &  &  &  & +z^{4} &  &  &  & +z^{8} &  &  & +\dots\\
&  &  &  &  &  & +z^{5} &  &  &  &  & +z^{10} & +\dots\\
&  &  &  &  &  &  & \ddots\\[3mm]
& = & z & +2z^{2} & +2z^{3} & +3z^{4} & +2z^{5} & +\dots\end{array}\]
We see the table houses a sieve of Eratosthenes and we find that \[
S(z)=z+2z^{2}+2z^{3}+3z^{4}+2z^{5}+4z^{6}+\cdots=\sum_{n\geq1}\tau(n)z^{n}\]
where $\tau(n)$ is the number of positive divisors of $n$. Again,
$S(z)$ has infinitely many singularities on the unit circle and is
not D-finite. 
\end{example}

\begin{example}\label{ex:comp}
Let $\mathfrak{L}_{comp}=\left\{ a^{c}\ :\ c\ \rm{is\ composite}\right\} $
denote the composite numbers written in unary. A generative grammar
is

\[
\mathbf{S}\to\mathbf{T}_{f\$}\qquad\mathbf{T}\to\mathbf{T}_{f}\vert\mathbf{R}\qquad\mathbf{R}\to\mathbf{RA}\vert\mathbf{AA}\]
\[
\mathbf{A}_{f}\to a\mathbf{A}\qquad\mathbf{A}_{\$}\to a\]
with sample derivation \[
\mathbf{S}\overset{*}{\to}\mathbf{R}_{f^{n}\$}\to\mathbf{R}_{f^{n}\$}\mathbf{A}_{f^{n}\$}\overset{*}{\to}\mathbf{R}_{f^{n}\$}\left(\mathbf{A}_{f^{n}\$}\right)^{m}\to\left(\mathbf{A}_{f^{n}\$}\right)^{m+1}\overset{*}{\to}a^{(n+1)(m+1)}.\]

This is certainly an ambiguous grammar because there is a separate,
distinct production of $a^{c}$ for each nontrivial factorization of
$c$. (Note: one can tweak the grammar to allow the trivial
factorizations $1\cdot c$ and $c\cdot1$. The resulting language
becomes the semigroup $a^{+}$ isomorphic to $\mathbb{Z}^{+}$ but the
generating function of all grammar productions is the familiar
$\sum\tau(n)z^{n}$ which we saw in Example~\ref{ex:divisors}.)
\end{example}

Suppose we want the generating function for the \emph{sum} of positive
divisors $\sum\sigma(n)z^{n}$? Then our table expansion above would
look like \[
\begin{array}{rcl@{ }l@{}l@{}l@{}l@{}l@{}l@{}l@{}l@{}l@{}l@{}}
 S(z) & = & z & +2z^{2} & +3z^{3} & +4z^{4} & +5z^{5} & +6z^{6} & +\cdots\\
 &  &  & +2z^{2} &  & +4z^{4} &  & +6z^{6} & +\cdots\\
 &  &  &  & +3z^{3} &  &  & +6z^{6} & +\cdots\\
 &  &  &  &  & +4z^{4} &  &  & +\dots\\
 &  &  &  &  &  & \ddots\\[3mm]
 & = & z & +4z^{2} & +6z^{3} & +12z^{4} & +\dots\end{array}\]
and now each row has closed form $\frac{z^{n}}{\left(1-z^{n}\right)^{2}}$.
Our goal is to modify the grammar of the previous example to obtain
$\sum\frac{z^{n}}{\left(1-z^{n}\right)^{2}}$. At first glance it
would seem that we only need replace the grammar rule $\mathbf{T}\to\mathbf{T}_{f}\vert\mathbf{A}_{f}\mathbf{R}_{f}$
with $\mathbf{T}\to\mathbf{T}_{f}\vert\mathbf{A}_{f}\mathbf{R}_{f}\mathbf{R}_{f}$
which replaces $S(z)=\sum A_{n}R_{n}$ with $\sum A_{n}R_{n}R_{n}$
. However this creates ambiguity because we can produce $ab$ in two
ways: $S\overset{*}{\to}A_{f\$}R_{f\$}R_{f\$}\overset{*}{\to}ab\varepsilon$
and $S\overset{*}{\to}A_{f\$}R_{f\$}R_{f\$}\overset{*}{\to}a\varepsilon b$.
An unambiguous solution is to \emph{create copies} $\mathbf{U}$
and $\mathbf{C}$ of the original grammar variables $\mathbf{B}$
and $\mathbf{R}$, respectively, so that $\mathbf{T}\to\mathbf{T}_{f}\vert\mathbf{A}_{f}\mathbf{R}_{f}\mathbf{U}_{f}$
is the replacement and we add new rules $\mathbf{U}\to\mathbf{UC}\vert\varepsilon$,
$\mathbf{C}_{f}\to c\mathbf{C}$ , and $\mathbf{C}_{\$}\to\varepsilon$.
The details are left to the reader, including how to re-write these changes in reduced form.

We conclude this section with the example of Bridson and Gilman~\cite{BrGi1996} 
alluded to in our introduction. They derive words that encode the 
\emph{cutting sequence} of the line segment 
from the origin to each integer lattice point in the Euclidean plane as the 
segment crosses the horizontal $(h)$ and vertical $(v)$ lines that join 
lattice points. 
Such sequences are made unique by declaring that as a segment passes through a
lattice point, the corresponding cutting sequence adds $hv$. For
instance, the cutting sequence for the segment ending at $(2,4)$ is the word
$hhvhhv$. 

\begin{example}\label{ex:BG} The grammar is given by
\[
\mathbf{S}\to\mathbf{T}_{\$}\qquad\mathbf{T}\to\mathbf{T}_{q}\vert\mathbf{T}_{r}
\vert\mathbf{U}_{q}\qquad\mathbf{U}\to\mathbf{VU}\vert\mathbf{V}
\qquad\mathbf{V}_{q}\to\mathbf{HV}\qquad\mathbf{V}_r\to\mathbf{V} \]
\[
\mathbf{V}_{\$}\to v\qquad \mathbf{H}_{q}\to\mathbf{H}\qquad
\mathbf{H}_{r}\to\mathbf{VH}\qquad\mathbf{H}_{\$}\to h\ .  \]
Attempting to solve the grammar equations immediately runs into
difficulty.  The valid sentential forms $\mathbf{V}_{qqrq\$}$ and
$\mathbf{V}_{qqqr\$}$ produce words of length eight and seven
respectively, which disallows the idea of simplification via Parikh
vectors. Indeed, a brute-force numerical attempt using the DSV method
has exponential time complexity in the length of index strings.

We circumvent this problem by introducing
two commuting formal variables $x,y$. Define $L(x,y)=\sum_{i,j>0}L_{i,j}x^iy^j$ 
where $L_{i,j}$ counts the number of cutting sequence words that have $i$ many
occurrences of $v$ and $j$ many $h$'s. The coefficients $L_{i,j}$ comprise
a frequency distribution on the first quadrant of the integer lattice. This
distribution contains more information than 
the one dimensional generating function $S(z)$. On the other 
hand, we can recover $S(z)=\sum_{n>1}v_nz^n$ by the formula 
$v_n=\sum_{i+j=n}L_{i,j}$. 

To compute the $L_{i,j}$, let us simplify the grammar by ignoring the 
context-free copying productions 
$\mathbf{U}\to\mathbf{VU}\vert\mathbf{V}$, change the
loading productions to $
\mathbf{T}\to\mathbf{T}_{q}\vert\mathbf{T}_{r}\vert\mathbf{V}_{q}$,
and begin by unloading sentential
forms $\mathbf{V}_{q\sigma}$, where $\sigma\in\left( q\vert r\right)^*\$ $. As per
the proof of Lemma~\ref{lem:balanced1} we define a \emph{step} as the 
application of the leftmost stack symbol to all non-terminals in a sentential 
form. If we start with an index string attached to $\mathbf{V}$
of length $l$, then after $n$ steps each 
non-terminal will have the same index string of length $l-n$.
For instance if we start with
$\mathbf{V}_{qqr\sigma}$ then the first step is
$\mathbf{V}_{qqr\sigma}\to\mathbf{H}_{qr\sigma}\mathbf{V}_{qr\sigma}$. The second step 
comprises two productions and ends with 
$\mathbf{H}_{r\sigma}\mathbf{H}_{r\sigma}\mathbf{V}_{r\sigma} $ while the third
step unloads the $r$ from each index and results in
$\mathbf{V}_\sigma \mathbf{H}_\sigma \mathbf{V}_\sigma \mathbf{H}_\sigma 
\mathbf{V}_\sigma $.

Let $x_i$ be the number of $\mathbf{V}$ non-terminals after performing 
step $i$, 
and let $y_i$ be the number of $\mathbf{H}$ 
non-terminals after performing step $i$. For a step unloading $q$ we observe 
the recursions $y_i = y_{i-1} + x_{i-1}$ and $x_i = x_{i-1}$. Likewise for $r$
we see recursions $y_i = y_{i-1}$ and $x_i = y_{i-1} + x_{i-1}$. 
Our simplified grammar always begins the unloading stage with 
$\mathbf{V}_{q\sigma}$ and thus we obtain the initial condition $x_1=1=y_1$
regardless of $\sigma$. 
This condition, along with our recursions above imply that for each $i$, the
pair of integers $(x_i,y_i)$ are relatively prime.

Suppose that $n$ is the last step needed to produce a cutting sequence
word from $\mathbf{V}_{q\sigma}$. Identify each pair $(x_n,y_n)$ with the
corresponding point in the integer lattice, so $x_n$ is the total number of
vertical lines crossed in the cutting sequence and $y_n$ is the number of
horizontal lines crossed. (Note that $x_n$ and $y_n$ depend on $\sigma$ as well
as $n$.)

We saw above that $\mathbf{V}_{qqr\$}\overset{*}{\to}\mathbf{V}_\$ 
\mathbf{H}_\$ \mathbf{V}_\$ \mathbf{H}_\$ \mathbf{V}_\$ 
\overset{*}{\to} vhvhv $ which is represented by $(3,2)$. The
pair $(2,3)$ is realized from $\mathbf{V}_{qrq\$}\overset{*}\to
hvhhv $. Indeed, the symmetry of the grammar implies that every generated 
pair $(x_n,y_n)$ has a generated mirror 
image $(y_n,x_n)$ obtained by transposing each $q$ and $r$ in the index
substring $\sigma$ attached to the initial unloading symbol 
$\mathbf{V}_{q\sigma}$.

We claim that every relatively prime pair of positive integers is realized as
$(x_n,y_n)$ for some cutting sequence word generated by our simplified
grammar. We show this by running in reverse the algorithm that generates cutting
sequences. Let $(i,j)\neq (1,1)$ denote a coprime pair and suppose by induction
that all other relatively prime pairs $(k,l)$ are the result of unique cutting
sequence words whenever ($k<i$ and $l\leq j$) or ($k\leq i$ and $l<j$), 
\emph{i.e.} whenever the point $(k,l)$ is strictly below or left of the 
point $(i,j)$. 
If $i<j$ then the letter $q$ was applied at the most recent step
with the previous pair being defined by $(i,j-i)$. 
On the other hand, if $i>j$ then the rightmost letter is $r$ and 
define the previous pair as $(i-j,j)$. In either case the new pair of
coordinates remain coprime and lie in the induction hypothesis zone. Note that
this is just the Euclidean algorithm for greatest common divisor and always 
terminates at $(1,1)$ when the starting pair $(i,j)$ are coprime.
Consequently  
the relatively prime pair $(i,j)$ is uniquely realized as
$(x_n,y_n)$ from some cutting sequence word $w$ generated by our simplified 
grammar. Furthermore $\mathbf{V}_{q\sigma}
\overset{*}{\to} w $ satisfies $\vert q\sigma\vert =n$ which
is the correct number of steps taken.

We apply our argument above to compute the two dimensional generating
function $L(x,y)$. For our
simplified grammar we have $L_{i,j}=1$ if the pair $(i,j)$ is relatively prime 
and $L_{i,j}$ vanishes otherwise. Equivalently, $L_{i,j}=1$ if and only if
$i$ and $i+j$ are coprime. To recover $S(z)=\sum_{n\geq 2}v_nz^n$ from $L(x,y)$ we
set $v_n=\sum_{i+j=n}L_{i,j}$, \emph{i.e.} we sum along slope $-1$ diagonal lines 
in quadrant one. Thus $v_n=\varphi (n)$ where $\varphi$ is Euler's totient 
function that counts
the number of integers $1\leq i<n$ that are coprime to $n$. Summary: we have 
successfully circumvented the exponential time complexity of computing
$S(z)=\sum_{\sigma\in (q\vert r)^*\$} V_{q\sigma}(z)$ and found that 
$S(z)=\sum_{n\geq 2}\varphi (n)z^n$.     

Recovering the original grammar by restoring the 
$\mathbf{U}$ productions allows for the construction of repeated
cutting sequences $w^t$, $w$ being the word associated to a coprime
lattice point $(i,j)$. This serves to add the lattice points 
$(ti,tj)$. Here we may assume $i$ and $j$ are relatively prime and $t\geq 2$
which makes these additions unique. (In fact, if $(ti,tj)=(sk,sl)$ for 
another coprime pair $(k,l)$, then both $s$ and $t$ are the 
greatest common divisor of the ordered 
pair and hence $s=t$.) The full grammar is in bijective correspondence with
the integer lattice strictly inside quadrant one. Words represent geodesics in 
the taxicab metric. Simple observation shows that the full growth series is represented by the rational
function 
\[ \sum_{n\geq 2}(n-1)z^n\ =\ \frac{z^2}{(1-z)^2} \]
and as a byproduct we have re-derived Euler's identity
\[ n-1 = \sum_{d|n, d>1}\varphi(d)   \qquad (!)\] 

\end{example}

\section{Ambiguity}
\label{sec:ambiguity}
We begin with the first published example of an inherently ambiguous 
context-free language~\cite{HoUl1979, Pari1966}. It
has an unambiguous indexed grammar. 
\begin{example}
Define $\mathfrak{L}_{amb}=\left\{ a^{i}b^{j}a^{k}b^{l}\::\ i,j,k,l\geq 1;\ i=k\
{\rm or} \ j=l\ \right\} $. The idea is to divide the language into a disjoint
union of the three sub-languages
\[ \mathfrak{L}_{X}=\left\{ a^{i}b^{j}a^{i}b^{l}\::\ j<l\ \right\} \quad
\mathfrak{L}_{Y}=\left\{ a^{i}b^{j}a^{i}b^{l}\::\ l<j\ \right\} \quad
\mathfrak{L}_{Z}=\left\{ a^{i}b^{j}a^{k}b^{j}\ \right\} \quad
\]
with no restrictions on $i,k$ other than that all exponents are at least one.
An indexed grammar is
\[
\mathbf{S}\rightarrow\mathbf{T}_{g\$}\qquad\mathbf{T}\rightarrow\mathbf{T}_{g}
\vert\mathbf{U}_{f}\vert\mathbf{Z}\qquad\mathbf{U}\to\mathbf{U}_{f}
\vert\mathbf{X}\vert\mathbf{Y}
\]

\[
\mathbf{X}\to\mathbf{ABAC} \qquad
\mathbf{Y}\to\mathbf{ACAB} \qquad
\mathbf{Z}\to\mathbf{DBEB}
\]

\[
\mathbf{A}_{f}\to a\mathbf{A}\qquad\mathbf{A}_{g}\to\varepsilon\qquad\mathbf{B}_{f}\to\mathbf{B}\qquad\mathbf{B}_{g}\to b\mathbf{B}\qquad\mathbf{B}_{\$}\to\varepsilon\]

\[
\mathbf{C}_{f}\to\mathbf{C}\qquad\mathbf{C}_{g}\to\ b\mathbf{C}\qquad
\mathbf{C}_{\$}\to b\mathbf{C}_{\$}\vert b
\qquad\mathbf{D}\to a\mathbf{D}\vert a\qquad\mathbf{E}\to a\mathbf{E}\vert a
\]
The reader is invited to draw the typical parse tree and verify that
the growth of this language is 
$\frac{z^{4}(1+3z)}{\left(1-z\right)^{3}\left(1+z\right)^{2}}$. 
\end{example}
Several other inherently ambiguous context-free language can be generated 
unambiguously
by an indexed grammar.
\begin{xca}
Another early example of an inherently 
ambiguous context-free language is featured in~\cite{ChSc1963}:  
$\mathfrak{L}=\left\{ a^{n}b^{m}c^{p}\::\ m=n>0\ {\rm or}\ m=p>0 \right\} $.
Write an unambiguous indexed grammar for it. Again, split the language into 
the disjoint union of three types of
words and build a grammar for each. The types are
$a^{n}b^{n}c^p$ with $0\leq p<n$, $a^{n}b^{n}c^p$ with $0<n<p$, and 
$a^{n}b^{p}c^p$ with $p>0$.
\end{xca}

Our examples beg the question: are there
inherently ambiguous indexed languages? Consider
Crestin's language of palindrome pairs defined by 
$\mathfrak{L}_{Crestin}=\left\{ vw\ :\ v,w\in\left(a\vert b\right)^{*},\
 v=v^{R}\ ,\ w=w^{R}\right\} $. It
is a {}``worst case'' example of an inherently ambiguous context-free
language (see~\cite{Flaj1987} and its references). We conjecture
that $\mathfrak{L}_{Crestin}$ remains inherently ambiguous as an
indexed language. What about inherently ambiguous languages that are
not context-free?

Consider the composite numbers written in unary as per Example~\ref{ex:comp}.
What would an unambiguous grammar for $\mathfrak{L}_{comp}$ look
like? We would need a unique factorization for each composite $c$. Since
the arithmetic that indexed grammars can simulate on unary output is discrete 
math (like 
addition and multiplication, no division or roots, etc), 
we need the Fundamental Theorem of Arithmetic. In fact, suppose
there is a different unique factorization scheme for the composites,
that doesn't involve a certain prime $p$. Then composite $c_{2}=p^{2}$
has only the factorization $1\cdot c_{2}$, and similarly $c_{3}=p^{3}$
has unique factorization $1\cdot c_{3}$ since $p\cdot c_{2}$ is
disallowed. But then $p^{6}=c_{2}\cdot c_{2}\cdot c_{2}=c_{3}\cdot c_{3}$
has no unique factorization. Therefore all primes $p$ are needed
for any unique factorization of the set of composites. Adding any
other building blocks to the set of primes ruins unique factorization.

Suppose we have an unambiguous indexed grammar for
$\mathfrak{L}_{comp}$.  It would be able to generate $a^{p^{k}}$ for
any prime $p$ and all $k>1$. This requires a copying mechanism (in the
manner of $\mathbf{R}$ in Examples~\ref{ex:serial} and~\ref{ex:comp})
and an encoding of $p$ into an index string (recall our slogan
{}``encode first, then copy'' from Section~\ref{sec:encodefirst}). In
other words, our supposed grammar for $\mathfrak{L}_{comp}$ must be
able to first produce its complement $\mathfrak{L}_{prime}$ and
\emph{encode these primes into index
 strings}. However,~\cite{PaDu1990} show that the set of index
strings associated to a non-terminal in an indexed grammar is
necessarily a regular language. On the other hand~\cite{Allen1968}
shows that the set of primes expressed in any base $m\geq 1$ does not form
a regular language. We find it highly unlikely that an indexed grammar
can decode all the primes from a regular set of index strings. We
conjecture that $\mathfrak{L}_{comp}=\left\{ a^{c}\ :\ c\ \mathrm{is\
   composite}\right\} $ is inherently ambiguous as an indexed
language.

Recall that a word is \emph{primitive} if it is not a power of another
word. In the copious literature on the subject it is customary to
let $\mathcal{Q}$ denote the language of primitive words over a two
letter alphabet. It is known that $\mathcal{Q}$ is not unambiguously
context-free (see ~\cite{Pete1994,Pete1996}, who exploits the 
original Chomsky-Sch\"utzenberger
theorem listed in Section 2 above). It is a widely believed conjecture
that $\mathcal{Q}$ is not context-free at all (see~\cite{DoHoItKaKa1994}).

$\mathfrak{L}'=\left\{ w^{k}\ :\ w\in\left(a\vert b\right)^{*},\
 k>1\right\} $ defines the complement of $\mathcal{Q}$ with respect
to the free monoid $\left(a\vert b\right)^{*}$. It is not difficult to
construct an ambiguous balanced grammar for $\mathfrak{L}'$ (a simple
modification of Example~\ref{ex:ww} will suffice). What about an
unambiguous grammar? Recall from~\cite{LySc1962} that
$w_{1}^{n}=w_{2}^{m}$ implies that each $w_{i}$ is a power of a common
word $v$. Thus to avoid ambiguity, each building block $w$ used to
construct $\mathfrak{L}'$ needs to be primitive. This means we must
not only be able to \emph{recreate} $\mathcal{Q}$ in order to generate
$\mathfrak{L}'$ unambiguously, we must be able to \emph{encode} each
word $w\in\mathcal{Q}$ as a string of index symbols, as per the
language of composites. We refer again to
Section~\ref{sec:encodefirst}. We find this highly unlikely and
we conjecture that $\mathfrak{L}'=\left\{ w^{k}\ :\ w\in\left(a\vert
   b\right)^{*},\ k>1\right\} $ is inherently ambiguous as an indexed
language.

\section{Open questions}

We observed that in many cases the generating function $S(z)$ of an indexed 
language
is an infinite sum (or multiple sums) of a family of functions related
by a finite depth recursion (or products/sums of the same). As we
mentioned earlier,~\cite{LK2003} give explicit sufficient conditions
for growth series of indexed languages on a unary alphabet. They also
show that such growth series are defined in terms of the recursions we
mentioned above.

Into what class do the generating functions of indexed languages fit?
Can we characterize the types of productions that lead to an infinite
number of singularities in the generating function? Ultimately, this
was a common property of many of the examples, and possibly one could
generalize the conditions that lead to an infinite number of
singularities in Example~\ref{ex:an2} to a general rule on production
types in the grammar. It seems that the foundation laid by Fratani and
Senizergues, in their work on catenative grammars is a very natural
starting point for such a study.

Can we characterize the expressive power of the grammars which satisfy
the hypotheses of Proposition~\ref{prop:push-to-inf II}? The alternate
characterizations of level 2 sequences in~\cite{Seni2007} might be
useful.  Can we show that $\{a^{p(n)}:p(n) \text{ is the $n$th prime}\}$
is a level 3 (or higher?) language? Perhaps this should precede a
search for an indexed grammar.

Is Crestin's language inherently ambiguous as an indexed language? What about
the composite numbers in unary or the complement of the primitive words?

In step with much modern automatic combinatorics, we would like to
build automated tools to handle these, and other questions related to
indexed grammars. Towards this goal the first author  has written a
parser/generator that inputs an indexed grammar and outputs words in
the language. It is licensed under GPLv3 and is available from the
authors. Is there a way to automate the process of pushing a set of
terminal variables off to infinity? 

Finally, we end where we began with the non-context-free language 
$\{a^{n}b^{n}c^{n}:n>0\}$.
It has a context-sensitive grammar
\[
\mathbf{S}\to abc\vert a\mathbf{BS}c\qquad\mathbf{B}a\to a\mathbf{B}\qquad b\mathbf{B}\to bb\]
for which the original DSV method works perfectly.
The method fails for several other grammars generating this same language.
What are the necessary and sufficient conditions to extend the method
to generic context-sensitive grammars? To what extent can we see these
conditions on the system of catenative recurrent relations?

\section*{Acknowledgements}
Daniel Jorgensen contributed ideas to the first two authors. The third
author offers thanks to Mike Zabrocki and the participants of the
Algebraic Combinatorics seminar at the Fields Institute for Research
in Mathematical Science (Toronto, Canada) for early discussions on
this theme. We are also grateful to the referees for pointing us to
several references. 


\end{document}